\documentclass[reqno,english]{amsart}
\usepackage{etex}
\usepackage{amsmath,amssymb,amsthm,bbm,mathtools,comment}
\usepackage[shortlabels]{enumitem}
\usepackage[pdftex,colorlinks,backref=page,citecolor=blue]{hyperref}
\usepackage[mathscr]{euscript}
\usepackage{tikz,babel,adjustbox}
\usepackage{cmtiup}
\usepackage{amsfonts}
\usepackage{caption}
\usepackage{verbatim}
\usepackage{array}
\usepackage[frame,cmtip,arrow,matrix,line,graph,curve]{xy}
\usepackage{graphpap, color, pstricks}
\usepackage{pifont}
\usepackage{cmtiup}

\allowdisplaybreaks

\theoremstyle{plain}
\newtheorem{theorem}{Theorem}[section]
\newtheorem{lemma}[theorem]{Lemma}

\newtheorem{prop}[theorem]{Proposition}

\newtheorem{conjecture}[theorem]{Conjecture}

\newtheorem{remark}[theorem]{Remark}
\theoremstyle{remark}

\newcommand{\pr}{\mathbb{P}}

\newcommand{\sub}[0]{\subseteq}

\newcommand{\beq}[1]{\begin{equation}\label{#1}}
\newcommand{\enq}[0]{\end{equation}}

\newcommand{\ra}[0]{\rightarrow}

\newcommand{\gl}[0]{\lambda }

\newcommand{\mn}[0]{\medskip\noindent}
\newcommand{\nin}[0]{\noindent}

\newcommand{\cE}{\mathcal{E} }
\newcommand{\cF}{\mathcal{F} }
\newcommand{\cG}{\mathcal{G} }
\newcommand{\cH}{\mathcal{H} }

\newcommand{\cU}{\mathcal{U} }

\title{A Proof of the Kahn--Kalai Conjecture}
\author{Jinyoung Park \and Huy Tuan Pham}

\email{jinypark@stanford.edu, huypham@stanford.edu}
\address{Department of Mathematics, Stanford University\\
450 Jane Stanford Way, Building 380, Stanford, CA 94305}

\begin{document}

\maketitle
\begin{abstract}

Proving the ``expectation-threshold'' conjecture of Kahn and Kalai, we show that for any increasing property $\cF$ on a finite set $X$, 
\[p_c(\cF)=O(q(\cF)\log \ell(\cF)),\]
where $p_c(\cF)$ and $q(\cF)$ are the threshold and ``expectation threshold'' of $\cF$, and $\ell(\cF)$ is the maximum of $2$ and the maximum size of a minimal member of $\cF$. 

\end{abstract}
\maketitle

\section{Introduction}\label{sec.intro}

Given a finite set $X$, write $2^X$ for the power set of $X$. For $p \in [0,1]$, let $\mu_p$ be the product measure on $2^X$ given by $\mu_p(A)=p^{|A|}(1-p)^{|X \setminus A|}$. In this paper $\cF \sub 2^X$ always denotes an \textit{increasing property}, meaning that if $B \supseteq A \in \cF$, then $ B \in \cF$. We say that an increasing property $\cF$ is nontrivial if $\cF \ne \emptyset, 2^X$. It is a well-known fact that $\mu_p(\cF) (:=\sum_{A \in \cF} \mu_p(A))$ is continuous and strictly increasing in $p$ for any nontrivial increasing property $\cF$. The \textit{threshold}, $p_c(\cF)$, is then the unique $p$ for which $\mu_p(\cF)=1/2$.  In this paper, we resolve a conjecture of Kahn and Kalai \cite{KK}, reiterated by Talagrand \cite{Talagrand}, relating the threshold and the "expectation-threshold" (definition is below).  

Following \cite{Talagrand}, we say $\cF$ is \textit{$p$-small} if there is $\cG \sub 2^X$ such that 
\beq{cover1} \cF \sub \langle \cG \rangle :=\bigcup_{S \in \cG} \{T: T \supseteq S\}\enq
and
\beq{p-small} \sum_{S \in \cG} p^{|S|} \le 1/2.\enq
We say that $\cG$ is a \textit{cover} of $\cF$ if \eqref{cover1} holds.
The \textit{expectation-threshold} of $\cF$, $q(\cF)$, is defined to be the maximum $p$ such that $\cF$ is $p$-small. Observe that $q(\cF)$ is a trivial lower bound on $p_c(\cF)$, since
\beq{thr.lb} \mu_p(\cF) \le \mu_p(\langle \cG \rangle) \le \sum_{S \in \cG} p^{|S|}.\enq
Note that, with $X_p$ the random variable whose distribution is $\mu_p$, the right-hand side of  \eqref{thr.lb} is $\mathbb E[|\{S \in \cG: S \sub X_p\}|]$.

Given an increasing property $\cF$, let $\ell_0(\cF)$ be the size of a largest minimal element of $\cF$, and let $\ell(\cF)=\max(\ell_0(\cF),2)$. Our main theorem resolves the following conjecture of Kahn and Kalai~\cite{KK}.

\begin{theorem}[The Kahn-Kalai Conjecture]\label{thm:KKC}
There is a universal constant $K$ such that for every finite set $X$ and nontrivial increasing property $\cF \sub 2^X$,
\[p_c(\cF) \le K q(\cF) \log \ell(\cF).\]
\end{theorem}

\nin Roughly speaking, Theorem \ref{thm:KKC} says that for any increasing property, the threshold is never far from its trivial lower bound given by the expectation threshold.

Thresholds have been a central subject in the study of random discrete structures since its initiation by Erd\H{o}s and R\'enyi \cite{ER1, ER2}, the study of which has flourished in random graph theory, computer science \cite{K, Talagrand}, and statistical physics \cite{Grimmett}. The definition of the threshold above is finer than the original Erd\H{o}s-R\'enyi notion, according to which $p^*=p^*(n)$ is \textbf{a} threshold for $\cF=\cF_n$ if $\mu_p(\cF) \rightarrow 0$ when $p \ll p^*$ and $\mu_p(\cF) \rightarrow 1$ when $p \gg p^*$. That $p_c(\cF)$ is always an Erd\H{o}s-R\'enyi threshold follows from \cite{BT}, in which it was observed that \textit{every} increasing $\cF$ admits a threshold in the Erd\H{o}s-R\'enyi sense.  While much work has been done identifying thresholds for specific properties (see \cite{Bollobas, JLR}), another intensively studied direction in the study of thresholds is ``sharpness'' of thresholds: we refer interested readers to \cite{Friedgut1, Friedgut2}.

To emphasize the strength of Theorem \ref{thm:KKC}, we point out that, in \cite{KK}, Kahn and Kalai wrote that ``It would probably be more sensible to conjecture
that it is not true.'' The expectation threshold is the most naive (and often the easiest) approach to estimating the threshold, and Theorem \ref{thm:KKC} says that for \textit{every} nontrivial increasing property, its threshold is only within a logarithmic factor of this naive estimate. In particular, many important works in this area have been on thresholds for \textit{specific properties,} and Theorem \ref{thm:KKC} easily implies some of those very hard results on the location of thresholds, for example, the appearance of perfect matchings in random $r$-uniform hypergraphs in seminal work of Johansson, Kahn and Vu \cite{JKV}, and the appearance of a given bounded degree
spanning tree in a random graph in recent work of Montgomery \cite{Montgomery} (we note that Montgomery shows a stronger ``universality'' result that above the correct threshold we actually have containment of all bounded degree spanning trees, which is not recoverable from Theorem \ref{thm:KKC}). For more discussion on the significance and applications of this theorem, we refer the readers to \cite{FKNP}, in which a weaker fractional relaxation of Theorem \ref{thm:KKC} was proved.

Two recent breakthroughs related to Theorem \ref{thm:KKC} are the significant progress on the Erd\H{o}s-Rado sunflower conjecture by Alweiss, Lovett, Wu and Zhang \cite{ALWZ} and the resolution of the fractional Kahn-Kalai conjecture by Frankston, Kahn, Narayanan and the first author \cite{FKNP}. The main lemma of \cite{ALWZ} gives a sufficient condition to guarantee that $X_p$ likely satisfies an increasing property, and is shown via a beautiful argument inspired from ideas in Razborov's proof \cite{R} of H\r{a}stad's switching lemma \cite{H}. In \cite{FKNP}, the authors observed the relevance of the developments in \cite{ALWZ} to thresholds and used an improved version of the argument in \cite{ALWZ}, achieved via separating typical ``non-pathological'' and atypical ``pathological'' cases, to establish Theorem \ref{CT'} below. Theorem \ref{CT'} is a sharper version of the main lemma of \cite{ALWZ} and a weaker version of Theorem \ref{thm:KKC}. 

We say $\cF$ is \textit{weakly $p$-small} if there is a $\gl:2^V \ra \mathbb R^+$ ($:=[0,\infty)$)
such that
\beq{wpsmall1}
\sum_{S\sub F}\gl_S\geq 1 ~~\textrm{for all } F\in \cF
\enq
and
\beq{wpsmall2} 
\sum_S\gl_Sp^{|S|} \leq 1/2.
\enq 
Note that if we restrict the range of $\gl$ to $\{0,1\}$, then we would recover the definition of the $p$-small property.
Define 
$
q_f(\cF)= \max\{\mbox{$p$ : $\cF$ is weakly $p$-small}\}
$ to be the \emph{fractional expectation-threshold} of $\cF$.
It follows from the definitions that
\beq{q's}
q(\cF)\leq q_f(\cF)\leq p_c(\cF),
\enq
and the main theorem of \cite{FKNP} resolves the following conjecture of Talagrand \cite[Conjecture 8.5]{Talagrand}, which is a weakening of Theorem \ref{thm:KKC}. 

\begin{theorem}\label{CT'}
There is a universal ${K}$ such that for every finite $X$ and nontrivial increasing
$\cF\sub 2^X$,
\[   
p_c(\cF)< Kq_f(\cF)\log \ell(\cF).
\]   
\end{theorem}

Talagrand \cite[Conjecture 6.3]{Talagrand} also conjectured that the parameters $q$ and $q_f$ are in fact always of the same order:

\begin{conjecture}\label{LT}
There is an absolute constant $L>0$ such that, for any nontrivial increasing $\cF$, $q(\cF) \ge q_f(\cF)/L$.
\end{conjecture}

\nin This of course implies equivalence of Theorems \ref{thm:KKC} and \ref{CT'}, and up until now, proving Conjecture \ref{LT} has been regarded as the most likely direction to prove Theorem \ref{thm:KKC}. However, as Talagrand observes, even simple instances of Conjecture \ref{LT} are not easy to establish. The two ``test cases'' suggested by Talagrand, which are now proved respectively in \cite{DK} and \cite{FKP}, already necessitate nontrivial arguments, and proving the full version of the conjecture is considered as a much harder task.

Our proof of Theorem \ref{thm:KKC} takes a surprisingly simple and direct approach rather than the indirect approach suggested by Conjecture \ref{LT}. Part of our proof is inspired by the argument in \cite{ALWZ, FKNP}, though our implementation is significantly different from the ideas in previous works \cite{ALWZ, FKNP, KNP}. In all of those papers, the notion of ``spread'' plays a key role in the proofs. In particular, it provides the starting point of the proof of Theorem~\ref{CT'} using linear programming duality, while in the setting of Theorem \ref{thm:KKC}, where one cannot exploit linear programming duality, this starting point immediately disappears. In the proof of Theorem~\ref{thm:KKC}, we completely avoid the use of spread, which is the essential workhorse of the proofs in \cite{ALWZ, FKNP, KNP}. This is one of the most surprising points of our proof. A key technical insight in our proof is the notion of \emph{minimum fragments}. With the minimum fragments, our proof allows to utilize a direct argument in the spirit of \cite{ALWZ}, thus giving a simplified proof of the main lemma of \cite{FKNP} as well as the main result of \cite{KNP} without separating pathological cases. It also suggests a clear intuition behind the linear relation between $p_c(\cF)$ and $q(\cF)$; see Remark \ref{rmk 2.2} and a recent paper by Bell and Frieze \cite{BF}.

The use of minimum fragments in the present paper is in fact inspired by our resolution \cite{PPT} of a conjecture of Talagrand (\cite[Problem 4.1]{Talagrand06}, \cite[Conjecture 5.7]{Talagrand} and \cite[Research Problem 13.2.3]{Talagrand2}) and a question of Talagrand (\cite{Talagrand-per} and a problem posed in \cite{Talagrand06}). The argument in \cite{PPT} uses a significantly more delicate and elaborate argument that shares some ideas with those in this paper. Importantly, even to address the weaker fractional relaxation version of these conjectures and questions of Talagrand, we need to use the full strength of the ideas in this paper as well as additional ideas in \cite{PPT}. 

\mn \textbf{Reformulation.} In Section \ref{sec.proofKK} we prove Theorem \ref{thm:KK} below, which implies Theorem \ref{thm:KKC}.  

A \textit{hypergraph} on $X$ is a collection $\cH$ of subsets of $X$, and a {set in the collection $\cH$} is called an \textit{edge} of $\cH$.  We say $\cH$ is \textit{$\ell$-bounded} if each of its edges has size at most $\ell$. Recall that $\langle \cH \rangle=\bigcup_{S \in \cH} \{T: T \supseteq S\}$. Note that we can extend  the definition of $p$-small to $\cH$ without any modification. For an integer $m$, we use an \textit{$m$-subset} of $X$ for a subset of $X$ of size $m$, and $X_m$ for a uniformly random $m$-subset of $X$.

\begin{theorem}\label{thm:KK}
Let $\ell\ge 2$. There is a universal constant $L$ such that for any nonempty $\ell$-bounded hypergraph $\cH$ on $X$ that is \textbf{not} $p$-small, 
\beq{conclusion} \mbox{a uniformly random $((Lp\log \ell)|X|)$-element subset of $X$ belongs to $\langle \cH\rangle$ with probability $1-o_{\ell\to \infty}(1)$.} \enq
\end{theorem}

\begin{remark}
It is easy to see from the proof that one can obtain a quantitative bound for the $o_{\ell\to \infty}(1)$ term of the form $(\log \ell)^{-c}$ for $c>0$. 
\end{remark}

\begin{proof}[Derivation of Theorem \ref{thm:KKC} from Theorem \ref{thm:KK}.]
Let $\cF$ be as in Theorem \ref{thm:KKC}. We assume Theorem \ref{thm:KK} and derive that if $q>q(\cF)$ then, with $p=Kq\log\ell(\cF)$ ($K$ is a universal constant to be determined), we have $\pr(X_p \in \cF) > 1/2$. Here, we recall that for $0\le p\le 1$, we denote by $X_p$ the random variable with distribution $\mu_p$.

Let $\cH$ be the set of minimal elements of $\cF$ (so $\langle \cH \rangle=\cF$). Then $\cH$ is $\ell(\cF)$-bounded and not $q$-small (since $q>q(\cF)$). Let $C$ be a (universal) constant for which, with $\ell=C\ell(\cF)$, the exceptional probability in Theorem \ref{thm:KK} is less than $1/4$. 

Let $m=(Lq\log \ell)|X|$ and $p'=2Lq\log \ell$, and choose $K$ sufficiently large so that $p\ge p'$. Theorem~\ref{thm:KKC} vacuously holds if $p'>1$, so we can assume that $p'\le 1$, in which case
\[
\pr(X_{p'} \in \langle \cH\rangle) \ge  \pr(X_m \in \langle \cH\rangle) \;\pr(|X_{p'}|\ge m) \ge (3/4)\; \pr(|X_{p'}| \ge m) > 1/2,
\]
where the last inequality follows from standard concentration bound, upon noting that $\cH$ is not $q$-small implies $|X|q>1/2$ (since $\{\{x\}:x \in X\}$ covers $\cH$) and hence $m > (L\log \ell)/2$ (so is somewhat large). This concludes the derivation.
\end{proof}
\nin \textbf{Notations and Conventions.} All logarithms are base $2$ unless specified otherwise. We have not made an attempt to optimize the absolute constants. For the sake of clarity of presentation, we omit floor and ceiling signs when they are not essential. 

\section{Proof of Theorem \ref{thm:KK}}\label{sec.proofKK}

Before going through the proof in detail, we first give an informal overview of our strategy. A hypergraph $\cH$ is $p$-small if $\cH$ admits a ``cheap'' cover, where being cheap refers to the condition in \eqref{p-small}. Our proof uses a randomized iterative process. Starting with $\cH_0=\cH$, at the $i$th step for $i\ge 1$, we assume that we have some hypergraph $\cH_{i-1}$ produced from the $(i-1)$th step for which $|S_{i-1}|\le 0.9^{i-1}\ell$ for all $S_{i-1}\in \cH_{i-1}$, and consider $W_i$ which is a random subset of $X\setminus(\bigcup_{1\le j<i}W_j)$ of size $w_i \approx Lpn$ (we will see later in the proof that for the purpose of producing a tail bound going to $0$ with $\ell$, it is helpful to choose slightly larger $w_i$ towards the last iterations). We will show that there is a sub-hypergraph $\cG_i$ of $\cH_{i-1}$ that admits a cover $\cU_i$ which has small cost with high probability, as well as a hypergraph $\cH_i$ with the following key properties:
\beq{reduction} \cH_{i-1} \setminus \cG_i \sub \langle \cH_i \rangle,\enq
\beq{capture} (\bigcup_{j\le i}W_j)\cup S_i \in \langle \cH\rangle \,\, \textrm{for all }S_i\in \cH_i, \textrm{ and}\enq
\beq{small} |S_i|\le 0.9^i\ell \,\, \textrm{for all }S_i\in \cH_i.\enq 
The choices of $\cG_i$, $\cU_i$ and $\cH_i$ depend on $W_i$. In the $(i+1)$th step we repeat our process with the hypergraph $\cH_i$. Property \eqref{reduction} reduces the task of finding a cover of the ``leftover'' $\cH_{i-1}\setminus \cG_i$ to finding a cover of $\cH_i$. Property \eqref{small} implies that for $\gamma > \log_{0.9}(1/\ell)$, either $\cH_{\gamma}$ is empty, or $\cH_{\gamma}$ only contains the empty set. In the former case, one can check that $\bigcup_i \cU_i$ is a cover of $\cH$; and in the latter case, one has from \eqref{capture} that $W=\bigcup_i W_i \in \langle \cH\rangle$. Since $\cH$ does not admit a cheap cover (as it does not satisfy \eqref{p-small}), and the cover $\bigcup_i \cU_i$ has small cost with high probability, we conclude that $W \in \langle \cH\rangle$ with high probability. Theorem \ref{thm:KK} then follows.

 In Section \ref{subsec.cover} we describe our construction of the cheap cover $\cU=\cU_i$ (in each step), and in Section~\ref{subsec.iteration} we analyze our iteration, concluding our proof.

\subsection{Constructing a cover}\label{subsec.cover} We use $n$ for $|X|$. Let $L$ {be large enough to support our conclusion} and let $\cH$ be $\ell$-bounded. In the following argument, we always assume that $S, S', S'', \hat S \in \cH$ and $W \in {X \choose w}$, where $w:= Lpn$ (as usual, ${X \choose w}$ is the collection of $w$-subsets of $X$).

The following notion of a \textit{minimum fragment} is key in our proof. Given $S$ and $W$, we say that $T=T(S,W)$ is a minimum $(S,W)$-fragment if $T$ is the set of the smallest size (possibly empty) such that $T$ can be written as $S' \setminus W$ for some $S' \in \cH$ with $S' \sub W \cup S$ (breaking ties arbitrarily). Note that such $S'$ exists as $S\in \cH$ and $S\sub W \cup S$. We use $t=t(S,W)$ for $|T(S,W)|$.

Given $W$, $\cG=\cG(W)$ is the collection of $S$ whose minimum fragment with respect to $W$ is ``large;'' formally,
\[\cG(W):=\{S \in \cH: t(S,W) \ge 0.9\ell\}.\]
Then we define $\cU(W)$, a cover of $\cG(W)$, as 
\[\cU(W):=\{T(S,W): S \in \cG(W)\}\]
(the fact that $\cU(W)$ covers $\cG(W)$ follows since each minimum fragment $T(S,W)$ is a subset of $S$).

Note that the edges in $\cH \setminus \cG(W)$ are not necessarily covered by $\cU(W)$. We define
\beq{H'} \cH'=\cH'(W)=  \{T(S,W): S \in \cH \setminus \cG(W)\};\enq
the hypergraph $\cH'$, which is $.9\ell$-bounded, will be the host hypergraph in the next iteration step (see \eqref{def.H'}). Note that $\cH \setminus \cG(W) \sub \langle \cH' \rangle$ (as promised in \eqref{reduction}), so in particular,
\beq{cover} \mbox{a cover of $\cH'$ also covers $\cH \setminus \cG(W)$.}\enq

The following lemma is our key lemma, which says that the cover $\cU(W)$ of $\cG(W)$ has small cost with high probability (over the randomness of $W$). 
\begin{lemma}\label{lem3.1} For $W$ uniformly chosen from ${X \choose w}$, 
\beq{3.1} \mathbb{E}\left[ \sum_{U\in \cU(W)}p^{|U|} \right] < L^{-0.8\ell},
\enq
where the expectation is over the randomness of $W$. 
\end{lemma}

Observe that Lemma \ref{lem3.1} is equivalent to
\beq{lem.alt} \sum_{W \in {X \choose w}} \sum_{U \in \cU(W)} p^{|U|} < {n \choose w}L^{-0.8\ell}. \enq

\begin{proof}[Proof of \eqref{lem.alt}]
Given $W$ and $m \ge 0.9\ell$, let
\[\cG_m(W):=\{S \in \cH: t(S,W) =m\}\]
and
\[\cU_m(W):=\{T(S,W): S \in \cG_m(W)\}.\]
Note that for any $U \in \cU_m(W)$ we have $|U|=m$, so $ \sum_{W \in {X \choose w}} \sum_{U \in \cU_m(W)} p^{|U|}$ is equal to $p^m$ multiplied by
\beq{TSW} \left|\left\{(W,T(S,W)): W \in {X \choose w}, S \in \cH, \mbox{ and } t(S,W)=m\right\}\right|.\enq
We bound the number of choices of $W$ and $T=T(S,W)$'s in the collection in \eqref{TSW} using the following specification steps.

\begin{enumerate}[Step 1.]
\item Pick $Z:=W \cup T$. Since $|Z|=w+m$ (note $W$ and $T$ are always disjoint), the number of possibilities for $Z$ is at most (recalling $w=Lpn$) 
\[{n \choose w+m} = {n \choose w} \cdot \prod_{j=0}^{m-1}\frac{n-w-j}{w+j+1} \le {n\choose w} (Lp)^{-m}.\]

\item
Note that $Z\; (= W \cup T)$ must contain an edge of $\cH$ by the definition of minimum fragment. Make a choice of $\hat S \sub Z$ arbitrarily so that the choice of $\hat S$ only depends on $Z$. In particular, the choice of $\hat S$ is free given $Z$. Here a crucial observation is that, since $T(S,W)$ is a minimum fragment,
\beq{T.in.S} T \sub \hat S.\enq
Indeed, since $\hat S$ is contained in $T\cup W\subseteq S\cup W$, the failure of \eqref{T.in.S} implies that $\hat S \subset S \cup W$ has $|\hat S \setminus W| < |T|$, contradicting the minimality of $T$.

The property \eqref{T.in.S} enables us to specify $T$ as a subset of $\hat S$, the number of possibilities of which is at most $2^\ell$.
\end{enumerate}
Note that $(W,T)$ is determined upon fixing a choice of $Z$ and $T$. In sum, we have
\[\sum_{W \in {X \choose w}} \sum_{U \in \cU_m(W)} p^{|U|}\le p^m{n \choose w}(Lp)^{-m}2^\ell={n \choose w}L^{-m}2^\ell,\]
and the left hand side of \eqref{lem.alt} is at most
\[\sum_{m \ge 0.9\ell} {n \choose w}L^{-m}2^\ell \le {n \choose w} L^{-0.8\ell}\]
for $L$ {sufficiently large}.
\end{proof}

\begin{remark}\label{rmk 2.2}
The proofs of the main lemma of \cite{ALWZ} and \cite{FKNP, KNP} also use similar-looking specification steps. By replacing their definition of $Z \; (:=W \cup S)$ with the one in the above proof $(Z:=W \cup T)$, one can remove the ``pathological'' case analysis in \cite{FKNP, KNP} and thus obtain a simplification of the argument there. 
\end{remark}

\subsection{Iteration}\label{subsec.iteration}
Recall that $n=|X|, \, \ell \rightarrow \infty,$ and $L$ is a large constant. Let $\gamma = \lfloor \log_{0.9}(1/\ell)\rfloor+1$. In the following definitions, $i=1, 2, \ldots, \gamma$. Let $\ell_i=0.9^i \ell$ and note that
\beq{gamma1}0.9 \le \ell_\gamma < 1. \enq

Let $X_0=X$ and $W_i$ be uniform from ${X_{i-1} \choose w_i}$, where $X_i=X_{i-1} \setminus W_i$ and $w_i=L_ipn$ with
\[L_i=\begin{cases}
L & \mbox{if } \quad i<\gamma - \sqrt{\log_{0.9}(1/\ell)}\\
L \sqrt{\log \ell}\; & \mbox{if } \quad \gamma - \sqrt{\log_{0.9}(1/\ell)} \le i \le \gamma.
\end{cases}\]
At the end, $W := \bigcup_{i=1}^{\gamma} W_i$ is a uniformly random $(CLp\log \ell)n$-subset of $X$ where $C \le C'$ for some absolute constant $C'>0$. 
Note that there is an absolute constant $c>0$ for which
\beq{i.bound} \mbox{$\ell_i >  \exp(c \sqrt{\log \ell}) \quad \textrm{for all } i < \gamma - \sqrt{\log_{0.9}(1/\ell)}$.}\enq

By iteratively applying our argument in Section \ref{subsec.cover}, we produce a sequence $\{\cH_i\}$ with $\cH_0=\cH$ and
\beq{def.H'} \cH_i=\cH_{i-1}' \enq
(see \eqref{H'} to recall the definition of $\cH'$). Note that each $\cH_i$ is $\ell_i$-bounded, and associated to each set $W_i$ in step $i$, we have $\cG_i=\cG_i(W_i) \subseteq \cH_{i-1}$ and a cover $\cU_i=\cU_i(W_i)$ of $\cG_i$. Properties \eqref{reduction}, \eqref{small}, and inductively Property \eqref{capture}, can be easily verified from our construction of $\cG_i$, $\cU_i$ and $\cH_i$. Indeed, assuming that Property \eqref{capture} holds for $\cH_{i-1}$, then, since each hyperedge $S_i \in \cH_i$ has the form $S_i = T(S_{i-1},W_i)=S_{i-1}'\setminus W_i$ for some $S_{i-1},S_{i-1}'\in \cH_{i-1}$, we have 
\[
(\bigcup_{j\le i}W_j)\cup S_i =(\bigcup_{j\le i-1}W_j) \cup W_i \cup (S_{i-1}'\setminus W_i) \supseteq (\bigcup_{j\le i-1}W_j) \cup S_{i-1}'\in \langle \cH\rangle.
\]

\begin{prop}\label{obs.term}
We have that either $W\in \langle \cH\rangle$ or $\cU:=\bigcup_{i \le \gamma} \cU(W_i)$ covers $\cH$.
\end{prop}

\begin{proof}
Note that $\ell_\gamma<1$, hence either $\cH_\gamma = \emptyset$ or $\cH_\gamma = \{\emptyset\}$. 

In the former case ($\cH_\gamma =\emptyset$), we show that $\cU$ covers $\cH$. Indeed, for $S\in \cH$, let $S=S_0, S_1, S_2, \ldots$ ($S_i \in \cH_i$) be the evolution of $S$ in the iteration process, i.e., $S_i:=T(S_{i-1},W_i)$. In each iteration, either $S_i \in \cG_i$ and is covered by $\cU_i$, or otherwise, $S_{i+1}\in \cH_{i+1}$. Since $\cH_\gamma=\emptyset$, there exists $j < \gamma$ for which $S_j \in \cG_j$. Hence, $S$ is covered by $\cU$. 

In the latter case ($\cH_\gamma = \{\emptyset\}$), we show that $W\in \langle \cH\rangle$. Indeed, by Property \eqref{capture}, $W_1\cup \ldots \cup W_\gamma\cup \emptyset \in \langle \cH\rangle$, and hence $W \in \langle \cH\rangle$. 
\end{proof}

Let $\cE$ be the event that $\cU$ covers $\cH$. By Proposition \ref{obs.term}, we have
\[\pr(W \in \langle \cH \rangle)+\pr(\cE) \ge 1;\]
equivalently,
\[\pr(W \in \langle \cH \rangle) \ge 1-\pr(\cE).\]
To bound $\pr(\cE)$, note that the expected ``cost'' for {the} cover $\cU$ is
\beq{exp.cost}
\begin{split}\mathbb E \left[\sum_{U\in \cU(W)} p^{|U|} \right] \stackrel{\eqref{3.1}}{<} \sum_{i \le {\gamma}} L_i^{-0.8\ell_i} & =\sum_{i< \gamma-\sqrt{\log_{0.9}(1/\ell)}} L_i^{-0.8\ell_i} +  \sum_{i= \gamma-\sqrt{\log_{0.9}(1/\ell)}}^{\gamma} L_i^{-0.8\ell_i}\\
& \stackrel{ \eqref{gamma1},\eqref{i.bound}}{\le} 2L^{-0.8\exp(c\sqrt{\log \ell})} + O((L \sqrt{\log \ell})^{-c'}) = (\log \ell)^{-c''}
\end{split}
\enq
for some constant $c',c''>0$. Also note that, by the assumption that $\cH$ is not $p$-small,
\beq{cost.large} \mbox{ if $\cE$ occurs, then  $\sum_{U\in \cU(W)} p^{|U|}>1/2$.}\enq
Therefore, by combining \eqref{cost.large} and Lemma \ref{lem3.1}, we have
\[\pr(\cE) \le \pr\left(\sum_{U\in \cU(W)} p^{|U|}>1/2 \right)  \stackrel{(\star)}{\le} 2 \mathbb{E}\left[ \sum_{U\in \cU(W)}p^{|U|} \right] \stackrel{\eqref{exp.cost}}{=}(\log \ell)^{-c''} = o_{\ell\to \infty}(1). \]
where $(\star)$ {follows from} Markov's Inequality. This gives the desired conclusion.

\section{Acknowledgements}
The authors are deeply grateful to Jacob Fox, Jeff Kahn, and David Conlon for their support and helpful comments on the paper. We also would like to thank Bhargav Narayanan, Wojciech Samotij, Lutz Warnke, the anonymous FOCS reviewers and the anonymous JAMS reviewers for helpful comments that improve the exposition of the paper. Jinyoung Park is supported by NSF grant DMS-2153844. Huy Tuan Pham is supported by a Two Sigma Fellowship.

\end{document}